\documentclass[12pt, a4paper]{amsart}
\usepackage{amsmath,amssymb,amscd,amsfonts}

\usepackage{ifthen}
\usepackage[T1]{fontenc}
\usepackage[utf8]{inputenc}
\usepackage[all]{xy}
\usepackage{graphicx}
\usepackage{enumerate}
\usepackage{xspace}
\usepackage{pdfsync}
\usepackage{epic}
\usepackage{dsfont}

\newtheorem{theorem}{Theorem}[section]
\newtheorem{remark}[theorem]{Remark}
\newtheorem{definition}[theorem]{Definition}
\newtheorem{lemma}[theorem]{Lemma}

\newtheorem{proposition}[theorem]{Proposition}

\newtheorem{corollary}[theorem]{Corollary}
\newtheorem{example}[theorem]{Example}

\newcommand\Q{{\mathbb{Q}}}

\def\Sym{\mathop{\rm Sym}\nolimits}

\def\cO{{\mathcal O}}

\def\cF{{\mathcal F}}
\def\cL{{\mathcal L}}

\def\cC{{\mathcal C}}

\let\wh=\widehat

\begin{document}
\title[]{Specialness and Isotriviality for Regular Algebraic Foliations}

\author{Ekaterina Amerik}
\address{Universit\'e Paris-Sud \\
Laboratoire de Math\'ematiques d'Orsay \\
91405 Orsay, France, and \\
National Research University Higher School of Economics\\
Laboratory of Algebraic Geometry and its Applications\\
Usacheva 6, 119048 Moscow, Russia}

\email{ekaterina.amerik@math.u-psud.fr}

\author{Fr\'ed\'eric Campana}
\address{Universit\'e Lorraine \\
 Institut Elie Cartan\\
Nancy \\ Institut Universitaire de France\\ and KIAS scholar,\ KIAS\\
85 Hoegiro, Dongdaemun-gu\\
Seoul 130-722, South Korea,}

\email{frederic.campana@univ-lorraine.fr}

\date{\today}

\maketitle

\ \ \ \ \ \ \ \ \ \ \ \ \ \ \ \ \ \ \ \ \ \ \ \ \ \ \ \ {\it A Jean-Pierre Demailly}

\tableofcontents


\begin{abstract} We show that an everywhere regular foliation $\cF$ with compact canonically polarized leaves on a quasi-projective manifold $X$ has isotrivial family of leaves when the orbifold base of this family is special. By \cite{BPW17}, the same proof 
works in the case where the leaves have trivial canonical bundle. The specialness condition means that the $p$-th exterior power of the logarithmic extension of its conormal bundle does not contain any rank-one subsheaf of maximal Kodaira dimension $p$, for any $p>0$. This condition is satisfied, for example, in the very particular case when the Kodaira dimension of the determinant of the Logarithmic extension of the conormal bundle vanishes. Motivating examples are given by the `algebraically coisotropic' submanifolds of irreducible hyperk\"ahler projective manifolds.
\end{abstract}


\section{Introduction}

Smooth algebraic families of canonically polarized manifolds 
over a smooth quasiprojective base have been intensively studied in recent years, starting from the work by 
Viehweg and Zuo \cite{V-Z}. 
Their main result states that if $f: X\to B$ is such a family and $\overline{B}$ a smooth compactification such that the complement $S$ of $B$
in  $\overline{B}$ is a normal crossing divisor, then some symmetric power of the Log-cotangent bundle of $B$ has an invertible subsheaf whose Kodaira dimension is at least the number of moduli $Var(f)$ of the family. Viehweg conjectured that the base of a family of maximal variation ($Var(f)=dim(B)$) must be of log-general type. This
conjecture is established in \cite{CP13} (but see \cite{CP15} for a simpler argument). 

A more general conjecture was stated in \cite{Ca07}, asserting that the family is isotrivial (that is, $Var(f)=0$) if $B$ is {\it special}, which roughly means that $B$ does not admit a map onto a positive-dimensional ``orbifold'' of general type. We do not recall the precise 
definition of a special quasi-projective manifold in this introduction, and just mention that $(\overline{B}, S)$ is special if its log-Kodaira dimension is zero. This isotriviality conjecture implies that the moduli map factors through the ``core map'' (see \cite{Ca07}), and so the variation can be maximal only if the core map is the identity map on $B$, which is then of Log-general type.

The isotriviality conjecture has been proved by Jabbusch and 
Kebekus in dimensions two and three (\cite{JK}, \cite{JK'}). B.Taji (\cite{T}) proved it in general, using \cite{CP13}.
A simplified version of Taji's proof, based on \cite{CP15}, can be found in \cite{Cl}.

We consider here, more generally, the case when the family $f: X\to B$ is not smooth but only 
quasi-smooth, that is, has only multiple fibers with smooth reduction as singularities; $B$ may then acquire quotient singularities. 

Such is the case when there is a smooth foliation ${\cF}$ on 
$X$ such that its leaves are fibers of $f$. The base $B$ then carries a natural orbifold structure coming from the multiple fibers
and one can ask whether the specialness of the orbifold 
base again implies the isotriviality of the family. The definition of the specialness of the orbifold base in this (mildly) singular context is part of the problem.

In this paper we give two equivalent definitions of the specialness of the orbifold base: first as a property of the relative
cotangent of the foliation in \S4, and then via multiple fibres of fibrations (in the spirit of \cite{Ca07}) in \S9. 
Using Viehweg-Zuo sheaves and \cite{CP15}, we prove that if $X$ is a 
connected quasi-projective complex manifold with an everywhere regular 
foliation $\cF$ with compact leaves which are canonically polarised, 
then the family of its leaves is isotrivial provided that its orbifold base is special. The first step of the proof, in \S8, is a `tautological' base change which produces a family with non-mutiple smooth fibres. 

The recent work \cite{BPW17}, Theorem 9.9, produces a Viehweg-Zuo sheaf on the base of a smooth family of polarized projective manifolds with 
trivial canonical bundle. Therefore our result remains true, with the same proof, when the fibres of $f$ have trivial (instead of ample) canonical bundle (see also \cite{CP15}, Theorem 8.2 which establishes the analogue of Viehweg conjecture in this case).

This isotriviality statement should hold for more general fibres, probably when their canonical 
bundle is semi-ample (cf. \cite{V-Z}) , or even pseudo-effective, as the work of Popa and Schnell \cite{PS} seems to indicate.

It would also be interesting to extend our result to the case when $X$ is a
quasi-K\"ahler complex manifold, that is, complement of a proper subvariety in a K\"ahler manifold.

It is a great pleasure for us to dedicate this paper to Jean-Pierre Demailly. The methods he has developed are
important for some of our main references, such as \cite{BPW17}, and his theorem \cite{D} provides a potential source of 
examples or counterexamples.


\section{Regular Algebraic Foliations. Compactification.}

Let $X$ be a connected complex manifold of complex dimension $n$, and  $\cF$ an \emph{everywhere regular} holomorphic foliation on $X$, of rank $r, 0<r<n.$ The foliation $\cF$ is called {\it algebraic} if all of its leaves are compact. 

In this paper $X$ is always assumed quasi-K\"ahler, that is, a non-empty Zariski open subset of a compact K\"ahler manifold $ \overline{X}$, and $\cF$ is assumed algebraic. 
In this case, using the compactness of the components of the Chow-Barlet space of analytic cycles on 
$\overline{X}$, we obtain a proper and connected holomorphic fibration $f:X\to B$ onto an irreducible normal complex space $B$ of dimension $n-r$ whose reduced fibres $F_b,\ b\in B,$ are exactly the leaves of $f$. Conversely, any such fibration $f:X\to B$ defines an algebraic (everywhere regular) foliation $\cF$ which is the 
saturation of the kernel of $df$ in the tangent bundle $T_X$. The order of the holonomy group 
along $F_b,\ b\in B,$ is also the multiplicity of $F_b$ as a scheme-theoretic fibre of $f$. This fibration is ``orbi-smooth'' in the sense that all of its scheme-theoretic fibres have smooth reduced support, and $B$ has
quotient singularities (see \cite{AC} for details). We choose a smooth compactification $(\overline{X},\overline{D})$ 
such that $\overline{D}=\overline{X}-X$ is a simple normal crossing divisor, in such a way that that this fibration extends to a holomorphic fibration $ \overline{f}: \overline{X}\to  \overline{B}$ with $ \overline{B}$ normal, and such that $ \overline{D}= \overline{f}^{-1}(E),$ where $E= \overline{B}-B$ is a divisor on $\overline{B}$.

Our aim is to give criteria under which an algebraic foliation is isotrivial, that is, all of its generic leaves are isomorphic. Our main result is Theorem \ref{tisot} below. We assume that $X$ is quasiprojective and that the leaves of $\cF$ are canonically polarized. In fact the same argument applies to the case when the canonical bundle of the leaves of $\cF$ is trivial (remark \ref{CY}). The criterion we give is expressed in terms of specialness  (see \S\ref{sspec}) of the Log-conormal sheaf of $\cF$, which we define in the next section. This property will be shown to be equivalent in \S\ref{sorbgeom} to another, more geometric property: the specialness of the orbifold base $(B, D_B)$ of the fibration $f$, defined in \S 9.

\section{The Log-conormal sheaf of $\cF$.}

Let $\cF$ be an everywhere regular foliation on the connected  quasi-K\"ahler manifold $X$. Let $ \overline{X},B, \overline{B}, f, \overline{f}$ be as above.

Define the rank $r$ subbundle $\Omega^1_{X/\cF}\subset \Omega^1_X$ as the kernel of the quotient map $\Omega^1_X\to \cF^*:=\Omega^1_{\cF}$. This bundle is called the conormal bundle of $\cF$. It is also the saturation inside $\Omega^1_X$ of $f^*(\Omega^1_{B^{reg}})$ (where $B^{reg}$ denotes the smooth part of $B$).. 

On the compactification $\overline{X}$, we define an extension $\Omega^1_{ \overline{X}/ \overline{\cF}}$ as $( \overline{f}^*(\Omega^1_{ \overline{B}^{reg}}))^{sat}$. Here the saturation is taken in the logarithmic 
cotangent bundle $\Omega^1_{\overline{X}}(Log(\overline{D}))$. In general, extending sheaves to the
compactification, we shall systematically consider their saturations in a suitable ``large'' locally free 
sheaf. The reason is that a saturated subsheaf of a locally free (or, more generally, reflexive) sheaf is normal (see for example \cite{OSS}, Lemma 1.1.16), so that a standard argument
involving a version of Hartogs' lemma applies to prove the birational invariance of certain spaces of sections.

So for any $m\geq 0$, we define $(\otimes^m\Omega^1_{ \overline{X}/ \overline{\cF}})^{sat}$ as the saturation of $\otimes^m\Omega^1_{ \overline{X}/ \overline{\cF}}$ inside $\otimes^m(\Omega^1_{\overline{X}}(Log(\overline{D})))$, and similarly for $Sym^m(\Omega^1_{ \overline{X}/ \overline{\cF}})^{sat}$.

To avoid too many heavy notations, we define ${\Omega^p_{ \overline{X}/ \overline{\cF}}}$ as being already 
saturated: 
$\Omega^p_{ \overline{X}/ \overline{\cF}}:=(\wedge^p(\Omega^1_{ \overline{X}/ \overline{\cF}}))^{sat},\forall p\geq 0$, where the saturation takes place in the locally free sheaf of logarithmic $p$-forms. By Hartogs' lemma,
the space of sections of ${\Omega^p_{ \overline{X}/ \overline{\cF}}}$ does not depend on the choice of the 
compactification.

\begin{definition}\label{conorm} For a non-singular algebraic foliation $\cF$ on a quasi-K\"ahler $X$ together with a suitable K\"ahler compactification $ \overline{f}: \overline{X}\to  \overline{B}$, 
$\Omega^1_{ \overline{X}/ \overline{\cF}}$ is called the Log-conormal sheaf of $\cF$.\end{definition}

The properties of the conormal sheaf we are interested in will be likewise independent on the chosen compactifications.

\medskip

Let now $g:B\dasharrow Y$ be a dominant rational map, extended to a rational map $ \overline{g}: \overline{B}\dasharrow  \overline{Y}$ on compactifications (one may suppose $\overline{Y}$ smooth though $B$ usually has 
some singularities).  Let $h=g\circ f,\ \overline{h} = \overline{g}\circ  \overline{f}: \overline{X}\to \overline{Y}$.  Set $dim(Y)=p$, $0\leq p\leq  dim(B)=n-r$.

\medskip

The map $h$ induces a natural inclusion $h^*(K_Y)\subset \Omega^p_{X/\cF}:=\wedge^p\Omega^1_{X/\cF}$, as well as extensions $ \overline{h}^*:\otimes^m\Omega^1_{ \overline{Y}}\subset \otimes^m\Omega^1_{ \overline{X}/ \overline{\cF}},\forall m\geq 0$.

We consider the saturated inverse images by $\overline{h}^*$ of pluridifferentials on $\overline{Y}$: 
$(\overline{h}^*(\otimes^m\Omega^1_{ \overline{Y}}))^{sat}\subset (\otimes^m\Omega^1_{ \overline{X}/ \overline{\cF}})^{sat}$, and analogously for the sheaf of symmetric differentials and $\Omega^p_{\overline{Y}}$.

The Hartogs' lemma gives the following:

\medskip

\begin{lemma}\label{birinv} For any $g:B\dasharrow Y$, and any $m\geq 0$, $h^0(\overline{X}, (\overline{h}^*(\otimes^m\Omega^1_{ \overline{Y}}))^{sat})$ does not depend on the choices of $\overline{X}, \overline{D}, \overline{B}$. The same property holds for $h^0(\overline{X}, (\overline{h}^*(Sym^m\Omega^1_{ \overline{Y}}))^{sat})$ and $h^0(\overline{X},(\overline{h}^*\Omega^p_{ \overline{Y}})^{sat}), \forall p>0.$
\end{lemma}

\begin{definition} Let $\overline{X}, \overline{D}$ be as above, and $\cL\subset \otimes^m \Omega^1_{\overline{X}}(Log( \overline{D}))$ be a rank-one coherent subsheaf. 

Define: $\kappa^{sat}(\overline{X},\cL):=limsup_{k\to +\infty} \frac{Log(h^0(\overline{X},\cL^{\otimes k,sat}))}{Log\ k}$, the  saturation $\cL^{\otimes k,sat}$ of $\cL^{\otimes k}$ being taken in $\otimes^{mk}\Omega^1_{\overline{X}}(Log(\overline{D}))$.
\end{definition}

By the same principle as in \ref{birinv}, we see that $\kappa^{sat}(\overline{X},\cL)$ is independent from the birational model $(\overline{X}, \overline{D})$ chosen; more precisely, $\kappa^{sat}(\overline{X},\cL)$
is equal to the $\kappa^{sat}$ of the direct or inverse image of $\cL$ on a modification of  
$(\overline{X}, \overline{D})$. 

It therefore makes sense to consider the restriction of $\cL$ to $X$
and talk of $\kappa^{sat}(X,\cL)$.

\medskip

We shall also need the following elementary lemma.

\begin{lemma}\label{factor} Let $\overline{h}:\overline{X}\dasharrow \overline{Y}$ be a meromorphic, dominant, and connected fibration with $p=dim(\overline{Y})>0$. Then $\overline{h}^*(K_{\overline{Y}})\subset \Omega^p_{ \overline{X}/ \overline{\cF}}$ (as subsheaves of $\Omega^p(Log D)$) if and only if $\overline{h}$ factors through $\overline{f}$ (ie: if there exists $\overline{g}:\overline{B}\dasharrow \overline{Y}$ such that $\overline{h}=\overline{g}\circ \overline{f}$). \end{lemma}

{\it Proof}: In restriction to the open part of $X$ where $h$ is defined and submersive this is classical, and the statement 
over the compactification follows from the fact that $\Omega^p_{\overline{X}/\overline{\cF}}$ is saturated, 
so that the inclusion over the open part implies the inclusion.


\section{Specialness}\label{sspec}

\begin{definition}\label{dspec} We say that the orbifold base of $f$ is special if, for every connected dominant rational map $g:B\dasharrow Y$ with $dim(Y)=p>0$, we have: 
$\kappa^{sat}( \overline{X}, \overline{h}^*(K_{ \overline{Y}}))<p$ (by Lemma \ref{factor}, the
saturation can be taken inside $\Omega^p_{ \overline{X}/ \overline{\cF}}$). This is independent from the choice of $\overline{X}, \overline{D}$, by Lemma \ref{birinv}.\end{definition}

The term ``orbifold base of $f$'' will be justified in \S \ref{sorbgeom} in the spirit of the general theory of orbifold pairs as in \cite{Ca07}. Since this theory is rather technical, we prefer to introduce some of our results
in this and the following section and postpone the proofs until later.

\smallskip

The specialness property will be shown in Theorem \ref{tspecorb} to be equivalent to other, apparently stronger properties:

\begin{theorem}\label{pspecequiv} The specialness of the orbifold base of $f$ is equivalent to each of the following properties:

1. for any $p>0$, and any coherent rank-one subsheaf $\cL\subset \Omega^p_{ \overline{X}/ \overline{\cF}}$, one has
$\kappa^{sat}( \overline{X}, \cL)<p$;

2. For any $g:B\dasharrow Y$ as in Definition \ref{dspec}, 
 $\kappa^{sat}(\overline{X},\cL)<p$ for any line bundle $\cL\subset \otimes^m \overline{h}^*(\Omega^1_{\overline{Y}})$.
\end{theorem}

\medskip

An important, although very particular example, where specialness holds, is as follows.

\begin{theorem}\label{tk=0} Assume that $\kappa(\overline{X}, det(\Omega^1_{ \overline{X}/ \overline{\cF}}))=0$, the orbifold base of $f$ is then special.
\end{theorem}

The proof of \ref{tk=0} follows from Theorem \ref{k=0spec}.


\section{Isotriviality criterion}

We can now formulate our main result.

\begin{theorem}\label{tisot} Let $f:X\to B$ be the fibration associated to an algebraic and everywhere regular foliation $\cF$ on the connected quasi-projective manifold $X$. Assume that the fibres of $f$ are canonically polarised and that the  orbifold base of $f$ is special. Then $f$ is isotrivial.
\end{theorem}

This answers positively a question raised in \cite{AC} for $X$ quasi-projective (instead of quasi-K\"ahler there). It is quite likely that this more general case can be handled by similar arguments. The case when $X$ is compact and $\cF$ is of rank $1$ was treated in early versions of \cite{AC} but disappeared in the final version after a simplification of the proof of its main result. The case when $f$ is submersive was established in \cite{T}.

\begin{corollary}\label{corisot} Let $f:X\to B$ be the fibration associated to an algebraic and everywhere regular foliation $\cF$ on the compact connected quasi-projective manifold $X$. Assume that the fibres of $f$ are canonically polarised and that $\kappa( \overline{X},det(\Omega^1_{ \overline{X}/ \overline{\cF}}))=0$. Then $f$ is isotrivial.
\end{corollary}

\begin{remark}\label{CY} The same assertions hold in the case when the fibres of $f$ have trivial, rather than ample, canonical bundle. See remark \ref{explic-CY} in the next section.
\end{remark}


\section{Viehweg-Zuo sheaves}\label{svzs}

Let again $f:X\to B$ be the fibration associated to an everywhere regular and algebraic foliation on a connected quasi-K\"ahler manifold $X$. We assume here that its fibres are canonically polarised 
and have Hilbert-Samuel polynomial $P$. Let $Mod_P$ be the quasi-projective scheme constructed 
in \cite{V}, parametrising the manifolds which are canonically polarised with Hilbert-Samuel polynomial $P$. If $B^*\subset B$ is the (non-empty) Zariski open subset of points over which $f$ is submersive, there is a natural holomorphic map $\mu^*:B^*\to Mod_P$ sending $b$ to the isomorphism class of $F_b$. 

Its image $M$ is algebraic, of dimension $Var(f)\in \{0,1,...,dim(B)=n-r\}$, where $Var(f)$ is the generic rank of the Kodaira-Spencer map $KS: T_{B^*}\to R^1f_*(T_{X/B})$.

When $f$ is submersive, $B^*=B,\mu^*=\mu$, and $B$ is smooth. We can thus then choose compactifications such that $\overline{B}$ is smooth, and $S:=\overline{B}-B$ is of simple normal crossings.

We have the following result of Viehweg and Zuo (\cite{V-Z}). 

\begin{theorem}\label{vz} Assume that $f:X\to B$ is submersive. There exists a line bundle $\cL\subset Sym^m(\Omega^1_{\overline{B}}(Log(S)))$ such that $\kappa(\overline{B},\cL)\geq Var(f)=dim(M)$.

\end{theorem}

A refinement of Theorem \ref{vz} by Jabbusch and Kebekus ((\cite {JK}, Theorem 1.4) states that 
this $\cL$ actually comes from the moduli space: $\cL\subset \Sym^m(\mu^*(\Omega^1_M))^{sat}$ (by abuse of notation, we write $\mu^*$ for the image of $d\mu$; cf. section 3). We call such an $\cL$ {\it a Viehweg-Zuo sheaf}.  

\begin{remark}\label{explic-CY} Theorem 9.9 of \cite{BPW17} establishes the existence of a Viehweg-Zuo 
sheaf (that is, a rank-one subsheaf of the logarithmic symmetric differentials with log Kodaira dimension
equal to the variation of the family) on the base of a smooth quasiprojective family of projective 
manifolds with trivial canonical class. Such a base admits a natural map into the moduli space $Mod_H$ 
of polarized manifolds, constructed by Viehweg in \cite{V}. The argument of \cite{JK} goes through in 
this case once the existence of Viehweg-Zuo sheafs is established. Our subsequent considerations do not 
use the ampleness of the canonical class of the fibres, so that the results are also valid for 
quasiprojective
families of manifolds with trivial canonical class (see Remark \ref{CY}).
\end{remark}

 \medskip

In our setting of a fibration defined by a foliation, $f$ is not necessarily submersive. However we know that the only singular fibers of $f$
are multiple fibres with smooth reduction. Equivalently, the non-smoothness of the fibration is encoded in the finite, but nontrivial holonomy groups around the leaves of $\cF$. In the next two sections, we deal with this problem, recalling the Reeb stability theorem and providing a simple base-change to eliminate the multiple fibres. The new base then carries a Viehweg-Zuo sheaf. In section 9 we descend this sheaf back to the orbifold base of the original fibration and derive a contradiction with speciality in the non-isotrivial case.


\section{Reeb Stability Theorem}\label{srs}

Let again $\cF$ be a regular algebraic foliation on the complex manifold $X$. We know that all its holonomy groups 
are finite. In the $C^{\infty}$ category, Reeb stability theorem asserts that locally around a fiber $F$ with holonomy 
group $G$ and a local transverse $T$, $X$ is the quotient of $\tilde{F}\times T$, where $\tilde{F}$ is the $G$-covering
of $F$,  by the diagonal action of $G$, and the map $f$ is the projection to $T/G$. In the holomorphic situation,
the complex structure on the neighbouring fibers varies; however there is the following adaptation of Reeb stability
(see \cite{HV}). let $G_b$ be the (finite) holonomy group of $\cF$ along a fibre $F_b=f^{-1}(b),\ b\in B$. There exist an open neighborhood $b\in U\subset B$ and a finite Galois covering $\beta:U'\to U=U'/G_b$, such that the 
normalisation $X_{U'}$ of the fibered product $X_U\times_UU'$, where $X_U$ stands for $f^{-1}(U)$, is a $G_b$-\'etale covering
of $X_U$ and submersive over $U'$. The map $\beta:U'\to U$ is obtained by taking a smooth holomorphic local transverse
to (reduced) $F_b$; over a sufficiently small $U\subset B$ containing $b$ it is finite surjective.

Since the second projection $f':X_{U'}\to U'$ is submersive, it is $C^{\infty}$-equivalent to a product, so in the 
$C^{\infty}$ context one finds back the usual Reeb stability theorem . In particular, all fibres of $f$ are, up to finite \'etale equivalence, isomorphic as $C^{\infty}$-manifolds.

\section{Elimination of multiple fibres by base-change}\label{sbc}

Our generalisation is based on a simple trick (already introduced in \cite{AC} for fibrations in curves, but the general case is similar) which eliminates multiple fibres.

\medskip

Let $(X,\cF)$ be as above, $\cF$ algebraic and everywhere regular. Let $f:X\to B$ be the associated fibration. Let $f_X:X_X\to X$ be the fibration deduced from $f:X\to B$ by the base-change $\beta(=f):X\to B$, and normalisation of the fibre-product $F:X\times_BX\to X$, seen as the projection to the second factor, while $\gamma: X\times_B X\to X$ is the projection onto the first factor, and is seen as lying over $\beta:X\to B$. We thus have: $f_X=F\circ \nu$, where $\nu: X_X\to X\times_BX$ is the normalisation map.

\begin{lemma}\label{lbc} In the above situation, the fibration $f_X:X_X\to X$ is submersive.
\end{lemma}

\begin{proof} The fibration $F:X\times_BX\to X$ has a natural section given by the diagonal of $X$. The inverse image of this section has a unique component lying over $X$ which gives a section of the map $f_X:X_X\to X$, since $\nu$ is a finite map. Moreover, $X_X$ is smooth, as seen from Reeb stability theorem: indeed, if $U'$ is a germ of smooth manifold transversal to a fibre $F_b$ of $f$, and finite surjective over a neighborhood $U$ of $b\in B$, then the normalisation of $X\times_UU'$ is smooth. We can now write a neighborhood $V\in X$ of any point $x\in F_b$ in the form $U'\times W$, if $W$ is a neighborhood of $x$ in $F_b$. The fibration $f_b:X_X\to X$ is thus, by the same argument, a holomorphic submersion.
\end{proof}

\begin{lemma}\label{lbc'} In the above situation, the map $\mu:X\to Mod$ defined in \S\ref{svzs} factors through $B$.
\end{lemma}

\begin{proof} Let $b\in B$ be any point. Let $b\in U\subset B$ be any sufficiently small neighborhood, and let $\beta:U'\to U$ be the finite Galois cover of group $G$ defined by a germ of manifold $U'$ transverse to the reduction of the 
fiber $F_b$ as in \S \ref{srs}. Base-changing by $\beta$ and normalising, we obtain $\gamma:X'\to X$ and $f':X'\to U'$, $\gamma$ being $G$-Galois and \'etale, and $f'$ submersive. The map $\mu':U'\to Mod$ is well-defined and coincides with $\mu^*\circ \beta:U'\to Mod$, if $\mu^*:B^*\cap U=U^*\to Mod$ is defined as in \S\ref{srs}. Since $B$ is normal and $\beta:U'\to U$ finite and proper, the map $\mu^*:B^*\to Mod$ extends to $B$ as a holomorphic map $\mu:B\to Mod$. 
\end{proof}


\section{Orbifold geometry}\label{sorbgeom}

We shall actually prove a more detailed version of theorem \ref{pspecequiv}, namely Theorem \ref{tspecorb} below. Before this, some notions concerning the geometry of orbifold bases
need to be recalled.


\subsection{Orbifold bases}

We recall the set-up from \cite{Ca04} and \cite{Ca07}. An orbifold pair is a connected normal compact complex-analytic variety $Z$ together with a Weil $\Q$-divisor $D=\sum_{j}c_j D_j$ where $D_j$ are the irreducible
components and the rational coefficients $c_j\in ]0,1]$. The union $\lceil D\rceil=\cup_{j}D_j$ is called the {\it support}, or ``round-up'' of $D$. The extreme cases are when $D=0$, or when $D=\lceil D\rceil,$ so $c_j=1\ \forall j$. 

If $F\subset Z$ is an irreducible Weil divisor not contained in $\lceil D\rceil$, we define its coefficient $c_D(F)$ in $D$ to be $0$. Thus $D=\sum_{F}c_D(F).F$, the sum running over all irreducible Weil divisors $F$ of $Z$.

We say that the orbifold pair $(Z,D)$ is smooth if $Z$ is smooth and the support of $D$ has only simple normal crossings. If moreover $D=\lceil D\rceil$, we say that we have a {\it smooth logarithmic pair}.

The purpose of introducing these objects here is to encode (and eliminate in codimension one) the multiple fibres of fibrations by means of ``virtual base changes''. The orbifold pair $(X,D)$ above may indeed be seen as a virtual ramified cover of $X$ ramifying to (rational) order $m_j=(1-c_j)^{-1}\in\ ]1,+\infty]$ over $D_j$, at least in codimension $1$. The (rational or $+\infty$) numbers $m_j$ will be called the {\it multiplicities} of $D$ along the $D_j's$. Conversely, $c_j=(1-\frac{1}{m_j})$, and $D=\sum_F(1-\frac{1}{m_D(F)})F, F$ running over all irreducible Weil divisors of $X$.

Alternatively, a pair $(X,D)$ interpolates between the projective case when $D=0$, and the quasi-projective case when $D=\lceil D\rceil$.

The main example of orbifold pairs (with integral or infinite multiplicities) comes from orbifold bases of fibrations:

\begin{definition} Let $f:Z\to Y$ be a surjective holomorphic proper map with connected fibres (that is,  a fibration) between normal connected complex spaces with $\Q$-factorial singularities. Fix an orbifold pair structure $(Z,D)$ on $Z$. 

For each irreducible Weil divisor $E\subset Y$, write $f^*(E)=\sum_k t_kF_k+R$, where $F_k$ runs through the irreducible Weil divisors of $Z$ mapped onto $E$ by $f$, while $R$ consists of the $f$-exceptional Weil divisors of $Z$ mapped into, but not onto, $E$.

Define the multiplicity $m_{f,D}(E)$ relative to $D$ of the generic fibre of $f$ over $E$ by the 
formula $m_{f,D}(E)= inf_{k}\{t_k m_{D}(F_k)\}$. 

The {\bf orbifold base} $(Y,D_{f,D})$ of $f$ is an orbifold pair where the divisor is defined by the following formula $$D_{f,D}=\sum_{E}(1-\frac{1}{m_{f,D}(E)})E$$ where $E$ ranges through the irreducible Weil divisors of $Y$. 

\end{definition} 
 
 This sum is finite since $m_{f,D}(E)=1$ unless either $t_k>1$, or $m_D(F_k)\neq 1$ for all $k$. If $D=0$, the multiplicity $m_f(E)=inf_{k}\{t_k\}$ is the multiplicity of the fiber over a general point of $E$ as considered in \cite{Ca04}.

Sometimes, when the data $(f, D)$ is clear from the context, we shall write simply 
$D_Y$ rather than $D_{f,D}$.



\subsection{Orbifold morphisms}

\begin{definition}\label{deforbmorph} (\cite{Ca07}) Let $f:X\to Z$ be a fibration between connected complex manifolds equipped with smooth orbifold structures $(X,D)$ and $(Z,D_Z)$. We say that $f$ is an orbifold morphism if, for any irreducible divisors $F\subset Z$ and $E\subset X$ such that $f(E)\subset F$, with $f^*(F)=tE+R$ where the support of $R$ does not contain $E$, one has $t m_D(E)\geq m_{D_Z}(F)$, where $m_D(E)$ (resp. $m_{D_Z}(F))$ is the multiplicity of $E$ in $D$ (resp. of $F$ in $D_Z)$. 

We shall say that $f$ is an orbifold birational equivalence if moreover $f$ is birational and $f_*(D)=D_Z$. 

\end{definition}

The following two simple situations provide examples. We leave the easy check to the reader.

\begin{example}\label{exorbequiv}
Let $u:(Z',D')\to (Z,D)$ be a proper bimeromorphic holomorphic map between connected complex manifolds $Z',Z$, equipped with orbifold divisors $D',D$ such that both orbifolds $(Z',D')$ and $(Z,D)$ are smooth, and moreover $u_*(D')=D$. Assume that all $u$-exceptional divisors of $Z'$ are equipped with the multiplicity $+\infty$. Then $f$ is an orbifold birational equivalence.
\end{example}

\begin{example}\label{exorbfib} Let $f:(X,D)\to (Z,D_Z)$ be as in definition \ref{deforbmorph} above. Assume that $(Z,D_Z)$ is the orbifold base of $f$. This does not imply in general that $f$ is an orbifold morphism. This will, however, be the case as soon as the multiplicities in $D$ of the $f$-exceptional divisors $E\subset X$ are sufficiently large; in particular
  when all these multiplicities are equal to $+\infty$. 
\end{example}

We shall need good bimeromorphic models of fibrations as in the proposition below. These are obtained using Raynaud's flattening theorem and Hironaka's desingularisation.

\begin{proposition}\label{propneat} (\cite{Ca07}, Proposition 4.10, p. 843). Let $(X_1,D_1)$ be a smooth orbifold pair, with $X_1$ projective\footnote{Compact K\"ahler would be sufficient.} connected. Let $h_1:X_1\to Z_1$ be a fibration
(or, more generally, a dominant meromorphic map with connected fibers). There exists a commutative diagram: 

$$\xymatrix{
(X,D)\ar[r]^{u}\ar[d]^{h}&(X_1,D_1)\ar[d]^{h_1}\\
(Z,D_{Z})\ar[r]^{v} & Z_1\\}$$

where $u,v$ are birational, and moreover the following holds:

1. $u:(X,D)\to (X_1,D_1)$ is a birational orbifold morphism.

2. $(X,D), (X_1,D_1), (Z,D_{Z})$ are smooth.

3. $(Z,D_Z)$ is the orbifold base of $h:(X,D)\to Z$ 

4. $h:(X,D)\to (Z,D_Z)$ is an orbifold morphism.

5. Every $h$-exceptional divisor of $X$ is also $u$-exceptional.
\end{proposition}


\subsection{Smooth orbifold bases of equidimensional  fibrations}

The notions of morphisms and birational equivalence for orbifold pairs are defined in the preceding subsections only for smooth orbifold pairs. The appropriate definitions are in general not available in the singular case, and the notion of a resolution of a (normal, quasi-projective, say) orbifold pair is not available either. The problem is as follows: one 
can introduce the notion of a smooth model of an orbifold as soon as the underlying manifold is $\Q$-factorial (and so it makes sense to talk about the pullback of a 
Weil divisor), but it is not clear whether any two such models are necessarily birational in the orbifold sense (see \cite{Ca07}, p. 832--833). 
However in the special case described below, we can introduce smooth orbifold pairs $(\overline{B},\overline{D_{\overline{B}}})$ which can be seen as resolutions of compactifications of the quasi-projective pairs $(B,D_B)$. The important property is that, for a given $(B,D_B)$,  all of these $(\overline{B},\overline{D_{\overline{B}}})$ are birationally equivalent in the orbifold sense (Corollary \ref{rbireq} below). Roughly speaking, the reason is that we don't introduce ``unexpected'' exceptional divisors by base change in this particular case.

\medskip

We consider a {\it smooth} quasi-projective complex manifold $X$ together with a projective fibration $f:X\to B$ onto a normal quasi-projective variety $B$. We assume that $f$ is {\it equidimensional}, so that its (connected) fibres $X_b$ are all of the same dimension $r$. In particular, this is the case if $f$ is the family of leaves of an everywhere regular foliation $\cF$ on $X$. 

Put the trivial orbifold structure (i.e. the zero divisor) on $X$ and let
$(B,D_B)$ be the orbifold base of $f:X\to B$. Take projective compactifications $\overline{B_1}, \overline{X_1}$ with 
the following properties: $f$ extends to $\overline{f_1}:\overline{X_1}\to \overline{B_1}$;  $\overline{X_1}$ is smooth;  $\overline{D_1}:=\overline{X_1}-X$ is a simple normal crossing divisor. Next, choose smooth modifications $\overline{X},\overline{B}$ of $\overline{X_1}, \overline{B_1}$, in such a way that $\overline{f_1}$ lifts to $\overline{f}:\overline{X}\to \overline{B}$, and moreover such that $\overline{D'}:=\overline{X}-X',\overline{D_B}:=\overline{B}-B'$ are simple normal 
crossing divisors, where $X'\subset \overline{X}, B'\subset \overline{B}$ denote the inverse images of $X,B$ in  $ \overline{X},  \overline{B}$ respectively. By further blow-ups of  $ \overline{X},  \overline{B}$, we can also assume that, moreover:

(1) The union of $\overline{D_B}$ and the inverse image $E_{\overline{B}}\cup D_{\overline{B}}$ of the (Zariski)-closure of $D_B\subset B$ in $\overline{B}$ is a simple normal crossings divisor. 
Here $E_{\overline{B}}$ is the (Zariski)-closure of the exceptional divisor of $B'\to B$, while $D_{\overline{B}}$ is the (Zariski)-closure of the strict transform of $D_B$ in $\overline{B}$.

(2) The union of $\overline{D'}$ and of the (Zariski)-closure $E$ of the exceptional divisor of $\chi:X'\to X$ is a simple normal crossings divisor. 

\medskip

The divisor from (2) defines the orbifold structure $(\overline{X}, \overline{D})$: we assign the 
multiplicity $+\infty$ (or equivalently: coefficient $1$) to all of its components. We equip $\overline{B}$ with the 
following 
orbifold divisor $\overline{D_{\overline{B}}}$: its support is the union described in (1), the multiplicities of the exceptional 
components $E_{\overline{B}}$ and of the border components $\overline{D_B}$ are $+\infty$, while each component of the closure of the strict transform of $D_B$ is assigned its multiplicity in the orbifold base of $f:X\to B$. Roughly speaking,
the ``old'' components come with their ``old'' multiplicities, whereas the ``new'' ones acquire infinite multiplicities.

\begin{definition}\label{dresolution} For a given $f:X\to (B,D_B)$ as above, we shall call any $\overline{f_1}:\overline{X_1}\to \overline{B_1}$ as above a compactification of $f$, and any $\overline{f}:(\overline{X},\overline{D})\to (\overline{B}, \overline{D_{\overline{B}}})$ a compactified resolution of $f:X\to (B,D_B)$.
\end{definition}

\begin{lemma} $\overline{f}:(\overline{X},\overline{D})\to (\overline{B},\overline{D_{\overline{B}}})$ is an orbifold morphism to its orbifold base (smooth by construction).
\end{lemma} 

\begin{proof} By the fact that the components of the boundaries $\overline{D}, \overline{D_B}$ of both $\overline{X},\overline{B}$ are all equipped with infinite multiplicities, it is sufficient to consider only divisors of $\overline{X}, \overline{B}$ which intersect the inverse images of $X,B$ respectively. Because $f:X\to B$ is equidimensional, the inverse image in $\overline{X}$ of any irreducible divisor $\overline{F}\subset \overline{B}$ which is $\beta$-exceptional, where $\beta:B'\to B$ is the natural birational map, is $\chi$-exceptional, where $\chi:X'\to X$ is the similar modification. Since all of these exceptional divisors are also equipped with the infinite multiplicity, the inequalities required for $\overline{f}$ to be an orbifold morphism are satisfied for these divisors. The remaining divisors for which these inequalities need to be checked are now the strict transforms in $\overline{B}$ of the components of $D_B$. But the multiplicities assigned to them being the same ones as in $D_B$ itself, the verification is trivial. 
\end{proof}

\begin{remark} \label{remorbequiv}. Since the closure in $\overline{X}$ of any component $C$ of the exceptional divisor of $X'\to X$ is, by definition, equipped with the multiplicity $+\infty$, and $f:X\to B$ has equidimensional fibres, the following properties for a divisor $E$ in $\overline{X}$ which is not contained in $\overline{X}-X'$ are equivalent:

(1) $C$ is $f\circ \chi$-exceptional.

(2) $C$ is equipped with the multiplicity $+\infty$ in $\overline{D}$.
\end{remark}

\begin{definition}\label{defdom} Let $f:(X,D)\to (Z,D_Z)$ and $f':(X',D')\to (Z',D_{Z'})$ be  fibrations between connected projective manifolds $X,Z,X',Z'$, equipped with orbifold divisors $D,D_Z,D',D_{Z'}$ respectively. We say that $f'$ dominates $f$ if there exists birational morphisms $u:X'\to X$, and $v:Z'\to Z$ such that $v\circ f'=f\circ u$ and $u_*(D')=D,v_*(D_{Z'})=D_Z$. 
\end{definition}

The next lemma is needed to show that all our compactified resolutions are orbifold birationally equivalent.

\begin{lemma}\label{lemorbequiv} Let $f:X\to B$ be an equidimensional fibration with $X,B$ quasi-projective, $X$ smooth and $B$ normal. Let $\overline{f}:(\overline{X},\overline{D})\to (\overline{B},\overline{D_{\overline{B}}})$ be an orbifold fibration constructed as above from $f:X\to B$. If $\overline{f'}:(\overline{X'},\overline{D'})\to (\overline{B'},\overline{D_{\overline{B'}}})$ is another such fibration (i.e. obtained by the same construction), then:

1. There exists a third fibration $\overline{f''}:(\overline{X''},\overline{D''})\to (\overline{B''},\overline{D_{\overline{B''}}})$ arising from this construction, and dominating the first two ones.  

2. The domination maps $u:(\overline{X''},\overline{D"})\to (\overline{X},\overline{D})$ and 
$v:(\overline{B''},\overline{D_{\overline{B''}}})\to (\overline{B},\overline{D_{\overline{B}}})$ such that $v\circ \overline{f''}=\overline{f}\circ u$ are both orbifold birational equivalences, and the same for $u',v'$.
\end{lemma}

\begin{proof} The existence of $\overline{f''}:(\overline{X''},\overline{D''})\to (\overline{B''},\overline{D_{\overline{B''}}})$ dominating both $\overline{f},\overline{f'}$ is obtained by applying the construction of $\overline{f}$ to a fibration $\overline{f_1''}:\overline{X_1''}\to \overline{B_1''}$ which compactifies $f:X\to B$ and dominates both initial compactifications $\overline{f_1}:\overline{X_1}\to \overline{B_1}$ and $\overline{f'_1}:\overline{X'_1}\to \overline{B'_1}$. The fact that $u,v,u',v'$ are orbifold morphisms and thus orbifold birational equivalences now follows from Remark \ref{remorbequiv}.

\end{proof}

\begin{corollary}\label{rbireq} For a given $f:X\to B$, the smooth pairs $(\overline{B}, \overline{D_{\overline{B}}})$ are all birationally equivalent in the orbifold sense, and may be seen as orbifold resolutions of compactifications of $(B,D_B)$. 
\end{corollary}


\subsection{Integral parts of orbifold tensors} We recall the construction of orbifold differentials from \cite{Ca07}. Consider a smooth orbifold pair $(Z,D)$ and local analytic coordinates $(z)=(z_1,...,z_n)$ near a given point $a\in Z$, centered at $a$ and `adapted' to $D$, that is such that the support of $D$ is contained in the union of the coordinate hyperplanes in the domain of this chart. Thus $D$ has near $a$ an equation with fractional exponents: $0=\Pi_{j=1}^{j=n}z_j^{c_j}$.

Let $m>0$ be an integer. We then define $[T^m]\Omega^1(Z,D)$, also written 
$[\otimes^m]\Omega^1(Z,D),$ as the locally free subsheaf of $\cO_Z$-modules of 
$\otimes^m \Omega^1_Z(log(\lceil D\rceil))$ generated by the elements: $z^{-[c_J]}.dz_{j_1}\otimes ...\otimes dz_{j_m}$. Here $J$ runs over all multi-indices $(j_1,...,j_m)\in \{1,...,n\}^m$, and $z^{-[c_J]}=z_1^{-[k_1(J)c_1]}\dots z_n^{-[k_n(J)c_n]}$, where, for $j=1,...,n$, $k_j(J)$ is the number of occurrences of $j$ in $J$, that is, the number of $k\in \{1,\dots,m\}$ such that $j_k=j$.

So for instance $\frac{dz_1^{\otimes m}}{z_1^{[mc_1]}}$ is among the generators; if $m_1$ is the multiplicity of the corresponding componend of $D$
this is rewritten as $z_1^{\lceil m/m_1 \rceil}(\frac{dz_1}{z_1})^{\otimes m}$. When the multiplicities are infinite, we obtain the usual 
logarithmic differentials.

One can easily check that this sheaf is independent from the chosen adapted coordinates, and so well-defined globally. Moreover, it is also equal to the $G$-invariant part $[\pi_*(\otimes^m(\Omega^1(Z,D)))]^G$ of $\pi_*(\otimes^m(\Omega^1(Z,D)))$, if $\pi:Z_D\to Z$ is any $G$-Galois Kawamata cover adapted to $(Z,D)$ in the sense of \cite{CP15}.

The sheaves $[S^m](\Omega^1(Z,D))$ of symmetric orbifold differentials are defined as the (locally free, saturated) subsheaves of $[T^m](\Omega^1(Z,D))$ defined similarly by the obvious symmetrisation conditions. See \cite{Ca07} for an explicit description.

These tensors satisfy, just as in the case $D=0$, a bimeromorphic invariance property:

\begin{proposition}\label{pbirinv} (\cite{Ca07}, Theorem 3.5, p. 835) Let $u: (Z',D')\to (Z,D)$ be a bimeromorphic orbifold morphism. 

Then $u^*:H^0(Z,[T^m](\Omega^1(Z,D))\to H^0(Z',[T^m](\Omega^1(Z',D'))$ is an isomorphism, for each $m>0$. \end{proposition}

Although the proof (which is a simple application of Hartogs theorem) is given there for rank one subsheaves of the orbifold differential sheaves, it immediately implies the version given here.

\begin{definition}\label{dkappasat} (\cite{Ca07}) Let $(X,D)$ be a smooth orbifold pair with $X$ connected complex projective, of dimension $n$. Let $m>0$, and $\cL\subset [\otimes^m] \Omega^1(X,D)$ be a rank-one coherent subsheaf. For each integer $k\geq 0$, let $\cL^{\otimes k,sat}\subset [\otimes^{mk}]\Omega^1(X,D)$ be the saturation of $\cL^{\otimes k}$. We then define: $$\kappa^{sat}_{D}(X,\cL):=limsup_{k\to +\infty}\{\frac{Log( h^0(X,\cL^{\otimes k,sat}))}{Log (k)}\}\in \{-\infty,1,...,n\}.$$
\end{definition} 

As actually stated in \cite{Ca07}, Theorem 3.5, p. 835, we have the following birational invariance property for rank-one subsheafs 

\begin{proposition} \label{pbirinv'}: Let $u:(X',D')\to (X,D)$ be a morphism which is an orbifold birational equivalence between two smooth projective orbifolds. Let $\cL\subset [\otimes^m]\Omega^1(X,D)$ and $\cL'\subset [\otimes^m]\Omega^1(X',D')$ be rank-one coherent subsheaves. Assume that either $\cL':=u^*(\cL)$, or that $\cL=u_*(\cL')$.
Then: $$\kappa^{sat}_D(X,\cL)=\kappa^{sat}_{D'}(X',\cL').$$\end{proposition} 


\subsection{Lifting and descent of integral parts of orbifold tensors}

The following theorem shall be proved in the Appendix.

\begin{theorem}\label{updowndiff} Let $h:(X,D)\to (Z,D_Z)$ be a fibration which has the properties listed in Proposition \ref{propneat}. Let $m\geq 0$ be a fixed integer. To shorten the notations, write $E^m_X:=[T^m](\Omega^1(X,D))$, and $E^m_Z:=[T^m](\Omega^1((Z,D_Z))$. Then, for any $m\geq 0$:

1. $h^*(E^m_Z)\subset E_X^m$. 

2. Let $h^*(E^m_Z)^{sat}$ stand for the saturation of $h^*(E^m_Z)$ in $E^m_X$.
Then $h_*(h^*(E^m_Z))^{sat})=E^m_Z$.
\end{theorem}

\begin{corollary}\label{cuddiff} In the situation of Theorem \ref{updowndiff}, for some $m>0$, let $\cL_U\subset \otimes ^m \Omega^1_U$ be a rank-one subsheaf, where $U\subset Z$ is a dense Zariski-open subset. Let $\cL\subset[\otimes ^m](\Omega^1(X,D))$ be such that $\cL_{\vert h^{-1}(U)}=h^*(\cL_U)^{sat}$. 

If $\kappa^{sat}_D(X,\cL)\geq 0$, there exists a saturated rank-one subsheaf $\cL_Z\subset  [\otimes ^m](\Omega^1(Z,D_Z))$ such that $h^*(\cL_Z)\subset \cL^{sat}$, and $\kappa^{sat}_{D_Z}(Z,\cL_Z)=\kappa^{sat}_D(X,\cL)$.

In particular, if $\kappa^{sat}_D(X,\cL)=p=dim(Z)$, then $\kappa(Z,\cL_Z)=p$, with $\cL_Z\subset [\otimes ^m](\Omega^1(Z,D_Z))$.
\end{corollary}

\begin{proof} Indeed, one sets $\cL_Z=h_*(\cL)^{sat}$.
\end{proof}

In order to prove our isotriviality results, we need this corollary only in the special case when the 
multiplicities of $D$ are integral or infinite: indeed our orbifold structure arising from a foliation assigns integral multiplicities
to the components parameterizing the multiple fibers, and infinite multiplicities to the compactifying components. By construction it is clear that
passing to a smooth model we remain in the same special case. This case of Theorem \ref{updowndiff} and its corollary is proved in \cite{JK'}, Theorem 5.8, and 
our method here is similar; we postpone the proof to the Appendix and refer to \cite{JK'} for the time being. The main new ingredient of the proof is Lemma \ref{luddiff} permitting to deal with rational multiplicities.


\subsection{Special smooth orbifolds, proof of the isotriviality criteria.}

\begin{definition}\label{dspecorb}(\cite{Ca07}, Definition 8.1, Th\'eor\`eme 9.9) Let $(X,D)$ be a smooth connected projective\footnote{The definition makes sense in the compact K\"ahler, or even class $\cC$ case.} orbifold. We say that $(X,D)$ is `special' if, for any $p>0$, and any rank-one subsheaf $\cL\subset \Omega^p_X$, one has: $\kappa^{sat}_D(X,\cL)<p$.
\end{definition}

Let now $(X,D)$ be as in the preceding definition, and let $g:X\dasharrow Z$ be a rational dominant fibration onto a variety of dimension $p>0$ (which one may suppose smooth and projective).
We shall always implicitely replace $g:(X,D)\dasharrow Z$ by a birational smooth model $g':(X',D')\to (Z',D_{Z'})$ enjoying the properties 1-5 listed in Proposition \ref{propneat}. In order to simplify notations, we shall also denote $g:(X,D)\to (Z,D_Z)$ this new `neat' birational model. 

\begin{theorem}\label{tspecorb} Let $(X,D)$ be as in definition \ref{dspecorb}. The following properties are equivalent:

1. $(X,D)$ is special.

2. For any $p>0$ and any $g:X\dasharrow Z$ as above, $\kappa(Z,K_Z+D_Z)<p.$

3. For any $p>0$ and any $g:X\dasharrow Z$ as above, $\kappa^{sat}_D(X,g^*(K_Z))<p.$

4. For any $p>0,$ for any $m>0$, for any $g:X\dasharrow Z$ as above, and for any coherent rank-one $\cL_Z\subset \otimes^m \Omega^1_Z$, one has $\kappa^{sat}_D(X,g^*(\cL_Z))<p.$

5. For any $p>0,$ for any $m>0$, for any $g:X\dasharrow Z$ as above, and for any rank-one coherent $\cL\subset [\otimes^m] \Omega^1(X,D)$ such that $\cL_{\vert g^{-1}(U)}=g^*(\cL_{U})$ for some Zariski open subset $U\subset Z$ and  some $\cL_{U}\subset g^*(\otimes^m \Omega^1_U)$, one has $\kappa^{sat}_D(X,g^*(\cL_Z))<p$ for $\cL_Z$ as in Corollary \ref{cuddiff}. 
\end{theorem}

\begin{proof} The equivalence between properties 1,2,3 is established in \cite{Ca07}, Th\'eor\`emes 9.9 and 5.3. The implication $4\Rightarrow 3$ is immediate. The reverse implication follows from \cite{CP15}, Theorem 7.11, by a contradiction argument applied to $(Z,D_Z)$, together with Corollary \ref{cuddiff}, last assertion. The equivalence between properties 4 and 5 follows from Corollary \ref{cuddiff}. \end{proof}

An important example of special smooth orbifold is given by the following:

\begin{theorem}\label{k=0spec} (\cite{Ca07}, Th\'eor\`eme 7.7)) Let $(X,D)$ be a smooth projective connected orbifold such that $\kappa(X,K_X+D)=0$. Then $(X,D)$ is special.
\end{theorem}


\begin{corollary} Let $f:X\to B$ be a projective fibration between two connected quasi-projective varieties, $X$ smooth, $B$ normal. Assume that $f$ has equidimensional connected fibres. Let $\overline{f}:(\overline{X},\overline{D})\to (\overline{B}, \overline{D_{\overline{B}}})$ be any resolution of $f:X\to (B,D_B)$ (see definition \ref{dresolution} above). We shall say that $(B,D_B)$ is special if so is $(\overline{B}, \overline{D_{\overline{B}}})$. This does not depend on the choice of $\overline{f}:(\overline{X},\overline{D})\to (\overline{B}, \overline{D_{\overline{B}}})$.

We then have, for  $(\overline{B}, \overline{D_{\overline{B}}})$, the equivalence between the 5 properties listed in Theorem \ref{tspecorb}.

Assume in particular that $(\overline{B}, \overline{D_{\overline{B}}})$ is special. Let $g:\overline{B}\dasharrow Z$ be a fibration with $dim(Z)=p$, and assume the existence of $\cL\subset [\otimes^m]\Omega^1(\overline{X},\overline{D})$ with $\kappa^{sat}_{\overline{D}}(\overline{X},\cL)=p$.
If there is a $\cL_U\subset \otimes^m \Omega^1_{U}$ for some Zariski dense open subset $U\subset Z$ such that $\cL_{\vert (g\circ \overline{f})^{-1}(U)}=(g\circ \overline{f})^*(\cL_U)^{sat}$, then $p=0$. 
\end{corollary}

Notice that Lemma \ref{factor} implies that the specialness of $(B, D_B)$ in the sense of the last corollary is the same as the specialness of the orbifold base defined in \S \ref{sspec}.

\begin{corollary} Let $f:X\to B$ be the fibration associated to an everywhere regular and algebraic foliation $\cF$ on a connected quasi-projective manifold $X$. Assume that the fibres of $f$ are (projective and) canonically polarised (or assume that the fibres of $f$ have trivial canonical bundle). If the orbifold base $(B,D_B)$ of $f$ is special, then $f$ is isotrivial. 

If $\kappa(\overline{X}, det(\Omega^1_{\overline{X}/\overline{\cF}}))=0$, then $f$ is isotrivial. 
\end{corollary}

\begin{proof} Indeed, consider the smooth base-changed family over $X$ as in \S \ref{sbc}. There is a Viehweg-Zuo sheaf $\cL\subset [Sym^m]\Omega^1_{\overline{X}}(Log(\overline{D}))$ associated to this smooth family.  
By \cite{JK}, Theorem 1.8, this sheaf possesses the property of being generically lifted from a subsheaf 
of
$\Sym^m(\Omega^1_Z)$, where $Z$ is the 
(eventually compactified and 
modified) image of the moduli map $\mu:X\to Mod$ described in \S\ref{vz}, and its Kodaira dimension is equal to the 
dimension of $Z$. But by Lemma \ref{lbc'}, the map $\mu$ factors through $B$, and so generically on $B$ there is 
another subsheaf $\cL_U$ of the symmetric differentials which lifts to $\cL$ over an open subset. Now apply 
Corollary \ref{cuddiff} to extend it to the sheaf $\cL_{\overline{B}}$ of saturated Kodaira dimension equal to $dim(Z)$.
The speciality of $B$ implies $dim(Z)=0$. This establishes the first claim. The second one then follows from Theorem \ref{k=0spec}.
\end{proof}



\section{Two examples}

\subsection{Coisotropic submanifolds}

Let $X\subset Y$ be a compact complex submanifold of a compact connected K\"ahler manifold $Y$ of dimension $n=2m$ carrying a holomorphic symplectic $2$-form $s$. We say that $X$ is {\it coisotropic} (relatively to $s$) if, for any $x\in X$, the complex tangent space $T_xX$ to $X$ at $x$ contains its $s$-orthogonal. This defines an everywhere regular rank $r$ foliation $\cF$ on $X$, where $r$ is the codimension of $X$ in $Y$. This foliation is often called {\it characteristic foliation}. 

Every smooth divisor $X\subset Y$ is coisotropic, with $r=1$, so that it carries the characteristic foliation
of rank one. This was the case studied in \cite{AC}. 

If $X$ is coisotropic, we have: $2m-2r\geq 0$, and $dim(X)=2m-r\geq r=codim_Y(X)$. If $r=m$, $X$ is said to be Lagrangian. A somehow ``dual'' case is when $X$ is isotropic (that is, when $s$ vanishes on $T_xX\ \forall x\in X$). Thus Lagrangian means both isotropic and coisotropic.

Let $X\subset Y$ and $s$ be as above, with $X$ coisotropic. We say that $X$ is `algebraically coisotropic' if the characteristic foliation $\cF$ is algebraic. Such subvarieties appear in the study of ``subvarieties of constant cycles'' on holomorphically symplectic varieties, but one has to drop the smoothness assumption (see \cite{V}). 

One of our main motivations for this paper was to generalize the results of \cite{AC}, where we have proved 
that the fibration associated to the characteristic foliation on an algebraically coisotropic smooth divisor
is always isotrivial in the projective case, and deduced from this that on an irreducible holomorphically 
symplectic projective manifold $Y$, 
there are no non-uniruled smooth algebraically coisotropic divisors $X$ except in the trivial case when $Y$ is a 
K3 surface and $X$ is a curve.

The natural question for higher codimension is as follows: let $Y$ be an 
irreducible holomorphically symplectic manifold and $X\subset Y$ a non-uniruled algebraically coisotropic 
submanifold. Can one conclude that $X$ is lagrangian? 

\medskip

Our study provides some evidence for the affirmative answer, however the results are still extremely partial.
For instance, one has the following.

\begin{corollary} Let $X$ be a projective manifold of dimension $d$ with an everywhere regular algebraic foliation $\cF$ of rank $r$ whose leaves are canonically polarised (or have trivial canonical bundle). If $\cF=Ker(s)$, where $s$ is a section $s$ of $\Omega_X^{d-r}\otimes L$, with $L\in Pic(X)$ and $c_1(L)=0$, then $\cF$ is isotrivial. Moreover, $\kappa(X)=r$ in the canonically polarized case and $0$ in the trivial canonical bundle case.
\end{corollary} 

\begin{proof} Indeed, $det(\Omega^1_{X/\cF})=det(\Omega^1_{\overline{X}/\overline{\cF}})$ is then numerically trivial, since generated by $s$, and Theorem \ref{corisot} applies.
\end{proof}

A more specific example is the following (the case $r=1$ has been established in \cite{AC}). However 
 in this situation one can show, in the same way as in \cite{AC}, that the fibration associated to $\cF$ does not have
multiple fibers in codimension one, so that a simpler proof of isotriviality can be given.

\begin{example}\label{ex1} Let $X\subset Y$ be a connected projective coisotropic submanifold of codimension $r$ in a smooth projective manifold $Y$ equipped with a holomorphic symplectic $2$-form $s$. Let $\cF$ be the characteristic 
foliation on $X$ defined as $Ker(s^r)$. Assume that the leaves of $\cF$ are compact and canonically polarised. 
Then $\cF$ is isotrivial and $\kappa(X)=r$.

To answer the question raised above, one would need, e.g. in the case when $Y$ is irreducible hyperk\"ahler, 
a lower bound for Kodaira dimension of $X$: for instance $\kappa(X)\geq m$ would be sufficient to derive that $X$ is lagrangian. This is the approach from \cite{AC}, but we do not know whether it might work for higher-codimensional coisotropic subvarieties.
\end{example}

At this point we can obtain the answer only in some very particular cases.

\begin{example} In the situation of Example \ref{ex1}, assume that $X$ is of general type and $K_X$ is ample
in restriction to the leaves of $\cF$ (this is the case for instance when the normal bundle $N_{X/Y}$ is ample). Then $X$ is Lagrangian. Indeed: $\kappa(X)=dim(X)\geq m$.
\end{example}

\begin{example} In the above situation of Example \ref{ex1}, assume that $Y$ is a \it{simple} torus (rather than irreducible hyperk\"ahler). Then $X$ is Lagrangian. Indeed: $\kappa(X)=dim(X)$ since $Y$ is simple.
\end{example}

\subsection{Boundary of codimension at least $2$.}
We consider the following situation:  Let $X^+$ be an irreducible (not necessarily normal) complex projective variety of dimension $n$, let $X$ be the smooth locus of $X^+$. Assume that there exists on $X$ an everywhere non-zero $d$-closed holomorphic form $w$ of degree $m:=(n-r)$ defining an everywhere regular foliation $\cF:=Ker(u)$ with canonically polarised compact leaves of dimension $r$ on $X$, or with compact leaves with trivial canonical bundle. The $m$-form $w$ thus descends to a nowhere vanishing $m$-form $v$ on the smooth locus of $B$. Thus $v$ is a nowhere vanishing section of a suitable power $N$ of $K_B$, if $f:X\to B$ is the fibration associated to $\cF$, so that $B$ has only quotient singularities, and its canonical bundle is $\Q$-Cartier. Thus: $w=(f^*(v))^{\otimes N}$ is a generator of $(det(\Omega^1_{X/\cF}))^{\otimes N}$.

We shall assume also that $X^+\subset M$, where $M$ is a complex space such that $M^{reg}\cap X^+=X$, and that $w$ is the restriction to $X$ of a holomorphic $m$-form $\wh{w}$ on $M^{reg}$, which extends holomorphically on some (or any) resolution of the singularities of $M$. It follows that if $\delta:\overline{X}\to X^+$ is an arbitrary desingularisation, then $w$ extends to a holomorphic $m$-form $\overline{w}$ on $\overline{X}$ (by taking first an embedded resolution of the singularities of $X^+$, lifting $\wh{w}$, and then observing that the existence of $\overline{w}$ is independent of the resolution of $X^+$. It is actually sufficient for the existence of $\overline{w}$ that $w$ be induced in local embeddings of $X^+$, instead of a global one $X^+\subset M$).

\begin{proposition}\label{projective normal} Assume that $X^+,X,M, w$ are as in the above situation, and that $X=X^{+, reg}$ has complement in $X^+$ of codimension $2$ or more. If the leaves of $\cF$ on $X$ are compact and canonically polarised (or have trivial canonical bundle), then the family of leaves is isotrivial.
\end{proposition}

\begin{proof} Let $f:X\to B$ be the proper connected fibration associated to $\cF$ on $X$. This fibration extends naturally to a fibration $\overline{f}:\overline{X}\to \overline{B}$ where $\overline{B}$ is the normalisation of the (projective) closure in the Chow-Barlet space of $X^+$ of $f(X)$. Theorem \ref{tisot} shows that we only need to show that $\kappa:=\kappa(\overline{X},det(\Omega^1_{\overline{X}/\overline{\cF}}))=0$ to prove the claim. But the restriction to $X$ of $det(\Omega^1_{\overline{X}/\overline{\cF}})$ is $det(\Omega^1_{X/\cF})$, which is generated by $w$, and hence trivial. Because $w$ extends to $\overline{w}$, we have $\kappa\geq 0$. Let now $\overline{s}$ be a section of $det(\Omega^1_{\overline{X}/\overline{\cF}})^{\otimes m},$ for some $m>0$. Let $s$ be its restriction to $X$. The quotient $\varphi:=\frac{s}{w^m}$ thus defines a holomorphic function on $X$. Because $codim_{X^+}(X^+-X)\geq 2$, $\varphi$ extends as a holomorphic function on the normalisation of $X^+$, and is thus constant by compactness of $X^+$. Thus $\overline{s}=\varphi.\overline{w^m}$, and $\kappa=0$, as claimed.
\end{proof}

\begin{example}\label{coisot} Let $X^+$ be a divisor in a connected complex projective variety $M$ of dimension $2d=n+1$ equipped with a symplectic two-form $s$ on some of its resolutions. The form $u:=s^{d-1}$ satisfies the non-vanishing condition and defines an everywhere regular rank-one foliation $\cF$ on $X$. We can also, more generally, consider $X^+$ of codimension $r$ and coisotropic in the previous pair $(M,s)$, taking then $u=s^{d-r}$. The coisotropy condition means that $s$ has rank $r$ on $X$. \end{example}

\begin{corollary} Let $X^+\subset M$ be complex projective, irreducible, with $M^{2d}$ equipped with a holomorphic symplectic form $s$ as in Example \ref{coisot} above. Let $w:=s^{d-r}$. If $X=X^{+, reg}$ is coisotropic of codimension $r$, if $codim_{X^+}(X^+-X^{+,reg})\geq 2$, and if the foliation $\cF=Ker(w)$ has compact canonically polarised leaves on $X$ (or compact leaves with trivial canonical bundle), then $f$ is isotrivial.
\end{corollary}

\begin{example} Let $S$ be a $K3$-surface, $C\subset S$ a smooth connected projective curve of genus $g>1$, and $k\geq 2$ an integer. Let $q:S^k\to M:=S^k/S_k$, where $S_k$ is the permutation group acting on the factors be the quotient map. Let $Q:=q\circ j:C\times S^{k-1}\to M$ be the natural composition map, where $j:C\times S^k$ is the injection. Let $X^+:=Q(C\times S^{k-1})\subset M$ be its image. Let $\rho:S^{[k]}\to M$ be the Fogarty resolution by the Hilbert scheme. The preceding result applies to $X^+,X, w=s^{k-1}$, if $s$ is a symplectic form on $S^{[k]}$. Here the isotriviality is obvious by construction, but this shows that examples which satisfy our quite restrictive conditions do exist.\end{example}



\section{Appendix: proof of Theorem \ref{updowndiff}}

\begin{proof} The first claim of Theorem \ref{updowndiff} is proved in \cite{Ca07}, Proposition 2.11, p. 823. We thus check now the second claim.
Notice first that we need to check this claim only over the complement of a codimension $2$ subset $S$ of $Z$, because $E^m_Z$ is locally free.  In particular, we can assume that $h$ has equidimensional fibres over this complement. Finally, the $h$-horizontal part of $D$ (that is, the components of $D$ dominating $Z$) does not play any r\^ole here, since $f^*(E^m_Z)$ is saturated in $E^m_X$ over the locus $Z-S$, $S:=Supp(D_Z)$. Indeed, up to a Zariski-closed subset of codimension at least $2$ in $Z-S$, if $D^+=Supp(D)\subset X$, the fibration $h:(X,D^+)\to Z$ is smooth over $Z-S$ in the logarithmic sense, leading to an exact sequence (over $Z-S)$: $0\to h^*(\Omega^1_Z)\to \Omega^1(X,Log D^+)\to \Omega^1_{X/Z}(D^+)\to 0$ with torsionfree cokernel, implying the same property at the level of tensor powers, and a fortiori for $[T^m ]\Omega^1(X,D)\subset [T^m]\Omega^1(X,D^+)=\otimes^m(\Omega^1_X(Log D^+))$.

We may thus choose local coordinates $(z_1,z')$, $z':=(z_2,...,z_p)$ on $Z$, adapted to $D_Z$, and such that, locally on $Z$, $D_Z$ is supported on $Z_1$, the divisor of $Z$ of equation $z_1=0$, with $D_Z$-coefficient $c'=(1-\frac{1}{m.t})$, and such that for suitable local coordinates $x=(x_1,x'=(x_2,\dots,x_n))$ adapted to $D$ on a generic point of a component $D_1$ of $D$ such that $h^*(Z_1)=t_1.D_1+\dots$ in the local chart of $X$ considered, we have: $h(x)=(z_1=x_1^{t_1}, z_2=x_2,...,z_p=x_p)$. By the definition of the orbifold base of $h$, we also have: 

1. For some component $D'$ of $h^{-1}(Z_1)$, if $c=(1-\frac{1}{m})$ is the coefficient of $D'$ in $D$, and if $h^*(Z_1)=t.D'+\dots$, we have: the coefficient $c'$ of $Z_1$ in $D_Z$ is $c':=(1-\frac{1}{m.t})$, introduced above. Moreover:

2. $m''.t''\geq m.t$ for any other component $D''$ of $h^{-1}(Z_1)$, if $m'',t''$ are defined as for $D'$. This inequality holds in particular for $D_1, m_1,t_1$, with $m_1,t_1$ being the above invariants $m'',t''$ when $D'':=D_1$. 

Let now $w:=\frac{dz_1^{\otimes K}}{z_1^{[k.c']}}\otimes (dz')^{\otimes (m-K)}$ be any one of the generators of $[T^m](Z,D_Z)$, for some $0\leq k\leq m$. Here $K\subset \{1,\dots,m\}$ is a subset of cardinality $0\leq k\leq m$, $m-K$ its complement there, and $dz_1^{\otimes K}\otimes (dz')^{\otimes (m-K)}$ means the tensor product $dz_{j_1}\otimes\dots \otimes dz_{j_m}$, where $j_h=1$ if and only if $h\in K$, while $j_h\in \{2,\dots,n\}$ otherwise.

Computing, we get (up to a nonzero constant): $$h^*(w)=x_1^{t_1.(\lceil \frac{k}{m.t}\rceil)}.(\frac{dx_1}{x_1})^{\otimes k}.(dx')^{\otimes(m-k)}.$$

But $[T^m](X,D)$ contains the $\cO_X$-module $W_X$ generated by: $$w_X:=\frac{dx_1^{\otimes k}}{x_1^{[k.c_1)]}}.(dx')^{\otimes (m-k)}=x_1^{\lceil \frac{k}{m_1}\rceil}.(\frac{dx_1}{x_1})^{\otimes k}.(dx')^{\otimes(m-k)}.$$

The argument now mainly relies on the following elementary lemma, where $\lceil x\rceil,x\in \Bbb R$, denotes the `round-up' of $x$, that is the smallest integer greater or equal to $x$. One also has: $\lceil x\rceil=-[-x]$, where $[x]$ is the usual integral part:

\begin{lemma}\label{luddiff}  Let $t>0$ be an integer, and $x\in \Bbb R$. Then: 

(1) $t. \lceil\frac{x}{t}\rceil-\lceil x\rceil\in \{0,1,\dots,(t-1)\}$.

Let $t,t',m,m',x$ be positive real numbers, with $t,t'$ integers. Then:

(2) $N:=t'.\lceil\frac{x}{m.t}\rceil-\lceil \frac{x}{m'}\rceil\geq 0$ if $m'.t'\geq m.t,$ and:

(2') $N\in \{0,\dots,(t'-1)\}$ if $m.t=m'.t'$.

\end{lemma}\label{roundup}
\begin{proof} Claim (1): $ \lceil \frac{x}{t}\rceil= \frac{x}{t}+ \vartheta, \vartheta\in [0,1[$, thus $t.\lceil \frac{x}{t}\rceil=x+t.\vartheta$. 
Also: $\lceil x\rceil=x+\vartheta', \vartheta'\in [0,1[$. 
Thus: $t.\lceil \frac{x}{t}\rceil-\lceil x\rceil=t.\vartheta-\vartheta'\in ]t,-1[$ being an integer, we get the first claim.

Claim (2): $ t'.\lceil \frac{x}{m.t}\rceil= t'.\frac{x}{m.t}+ \vartheta= \frac{m'.t'}{m.t}.\frac{x}{m'}+ \vartheta, \vartheta\in [0,1[$. Moreover: $ \lceil \frac{x}{m'}\rceil= \frac{x}{m'}+ \vartheta', \vartheta'\in [0,1[$. Since $N:=t'.\lceil \frac{x}{m.t}\rceil-\lceil \frac{x}{m'}\rceil=(\frac{m'.t'}{m.t}-1). \frac{x}{m'}+t'.\vartheta-\vartheta'\geq t'.\vartheta-\vartheta'>-1$ is an integer, it is non-negative, as asserted.

Claim (2') follows from Claim (1), applied to $t',x':=\frac{x}{m'}$, in place of $t,x$, since: $\frac{x}{m.t}=\frac{x}{m'.t'}=\frac{x'}{t'}$.
\end{proof}

From Lemma \ref{roundup}, and since $m_1.t_1\geq m.t$, we get that $h^*(w)\in W_X$, and that $h^*(w)=x_1^{\tau}.w_X$, with $\tau \in \{0,\dots,(t_1-1)\}$ if $m_1.t_1=m.t$.

The support of $h_*(h^*(E^m_Z)^{sat})/E^m_Z$ must then have support of codimension one contained in $D_Z$. Assume that $Z_1$ for example is contained in this support. Then $h_*(h^*(E^m_Z)^{sat})_{\vert Z_1}\subset E^m_Z(k.Z_1)$ for some minimal integer $k\geq 0$. We will show that $k=0$, implying the claim. Assume $k\geq 1$, then $h^*(E^m_Z)$ vanishes at order $\tau\geq t_1$ on the component $D'$ of $h^{-1}(Z_1)$ introduced above, contradicting the inequality $\tau\leq (t_1-1)$ established in the previous lines. \end{proof}


\end{document}